\documentclass{amsart}
\usepackage[utf8]{inputenc}
\usepackage{amssymb}
\usepackage{hyperref}
\usepackage[final]{showkeys} %[final] is the option that removes them

\input xy
\xyoption{all}

\theoremstyle{definition}
\newtheorem{mydef}{Definition}[section]
\newtheorem{lem}[mydef]{Lemma}
\newtheorem{thm}[mydef]{Theorem}

\newtheorem{cor}[mydef]{Corollary}

\newtheorem{question}[mydef]{Question}

\newtheorem{defin}[mydef]{Definition}
\newtheorem{example}[mydef]{Example}
\newtheorem{remark}[mydef]{Remark}

\newtheorem{fact}[mydef]{Fact}

\newcommand{\cf}[1]{\text{cf} (#1)}
\newcommand{\seq}[1]{\langle #1 \rangle}
\newcommand{\rest}{\upharpoonright}

\newcommand{\s}{\mathfrak{s}}

\newcommand{\ts}{\mathfrak{t}}

\newcommand{\id}{\text{id}}

\newcommand{\leap}[1]{\le_{#1}}
\newcommand{\ltap}[1]{<_{#1}}

\newcommand{\geap}[1]{\ge_{#1}}

\newcommand{\lta}{\ltap{\K}}
\newcommand{\lea}{\leap{\K}}

\newcommand{\gea}{\geap{\K}}

\newcommand{\K}{\mathbf{K}}

\newbox\noforkbox \newdimen\forklinewidth
\forklinewidth=0.3pt \setbox0\hbox{$\textstyle\smile$}
\setbox1\hbox to \wd0{\hfil\vrule width \forklinewidth depth-2pt
 height 10pt \hfil}
\wd1=0 cm \setbox\noforkbox\hbox{\lower 2pt\box1\lower
2pt\box0\relax}
\def\unionstick{\mathop{\copy\noforkbox}\limits}

\def\1nf{\unionstick^{(1)}}

\def\2nf{\unionstick^{(2)}}
\def\3nf{\unionstick^{(3)}}

\newcommand{\gtp}{\mathbf{tp}}

\newcommand{\gS}{\mathbf{S}}

\newcommand{\Ii}{\mathbb{I}}

\newcommand{\Ll}{\mathbb{L}}

\newcommand{\cl}{\operatorname{cl}}

\newcommand{\LS}{\text{LS}}

\newcommand{\goodms}[1]{\text{good}^{-#1}}

\title{On categoricity in successive cardinals}
\date{\today \\
 AMS 2010 Subject Classification: Primary 03C48. Secondary: 03C45, 03C52, 03C55, 03C75.}
\keywords{abstract elementary classes, categoricity, amalgamation, no maximal models, good frames}

\author{Sebastien Vasey}
\email{sebv@math.harvard.edu}
\urladdr{http://math.harvard.edu/\textasciitilde sebv/}
\address{Department of Mathematics \\ Harvard University \\ Cambridge, Massachusetts, USA}

\begin{document}

\begin{abstract}
  We investigate, in ZFC, the behavior of abstract elementary classes (AECs) categorical in many successive small cardinals. We prove for example that a universal $\Ll_{\omega_1, \omega}$ sentence categorical on an end segment of cardinals below $\beth_\omega$ must be categorical also everywhere above $\beth_\omega$. This is done without any additional model-theoretic hypotheses (such as amalgamation or arbitrarily large models) and generalizes to the much broader framework of tame AECs with weak amalgamation and coherent sequences.
\end{abstract}

\maketitle

\tableofcontents

\section{Introduction}

The upward Löwenheim-Skolem theorem says that any first-order theory with an infinite model has models of arbitrarily large cardinalities. This result is no longer true outside of first-order logics, for example for theories in $\Ll_{\omega_1, \omega}$. For this more general case, it is reasonable to ask which properties of small models guarantee existence of bigger models. In that light, the following early result of Shelah \cite{sh88} is remarkable:

\begin{fact}\label{categ-ext}
  An $\Ll_{\omega_1, \omega}$ sentence which is categorical in both $\aleph_0$ and $\aleph_1$ has a model of size $\aleph_2$.
\end{fact}

Shelah proved in fact a more general theorem, valid for any ``reasonably definable'' abstract elementary class (AEC) with countable Löwenheim-Skolem number (see \cite[I.3.11]{shelahaecbook} for the details). In particular, he answered in the negative Baldwin's question: can a sentence in $\Ll (\exists^{\ge \aleph_1})$ have exactly one uncountable model? See for example \cite[\S4]{grossberg2002} for a more detailed history.

It is natural to ask whether Fact \ref{categ-ext} generalizes to \emph{any} AEC. More specifically:

\begin{question}\label{shelah-q}
  Assume $\K$ is an AEC with Löwenheim-Skolem-Tarski number $\lambda$. If $\K$ is categorical in both $\lambda$ and $\lambda^+$, must it have a model of cardinality $\lambda^{++}$?
\end{question}

Question \ref{shelah-q} is still open. Partial approximations immensely stimulated the field: Shelah \cite{sh576} has shown assuming some set-theoretic hypotheses that categoricity in \emph{three} successive cardinals suffices:

\begin{fact}
  Assume\footnote{Shelah originally proved this assuming in addition a saturation condition on the weak diamond ideal, but this was subsequently removed \cite[VI.8.1(3)]{shelahaecbook}.} $2^{\lambda} < 2^{\lambda^+} < 2^{\lambda^{++}}$. Let $\K$ be an AEC with Löwenheim-Skolem-Tarski number $\lambda$. If $\K$ is categorical in $\lambda$, $\lambda^+$, and $\lambda^{++}$, then it has a model of size $\lambda^{+3}$.
\end{fact}
\begin{remark}
  One of the byproduct of Shelah's proof is the machinery of \emph{good frames}, developed in Shelah's two volume book \cite{shelahaecbook}. Essentially, an AEC has a \emph{good $\lambda$-frame} if $\K_\lambda$, its class of models of cardinality $\lambda$, behaves ``well'' in the sense that it has several structural properties, including a superstable-like forking notion. Good frames have subsequently been used in many results, for example in the author's proof of the eventual categoricity conjecture for universal classes \cite{ap-universal-apal, categ-universal-2-selecta}, in the recent proof of the eventual categoricity conjecture from large cardinals \cite{multidim-v2}, and in the full analysis of the categoricity spectrum of AECs with amalgamation \cite{categ-amalg-selecta}.
\end{remark}

More ambitiously, one can ask when it is possible not only to prove the existence of a model from successive categoricity, but to prove the existence of \emph{arbitrarily large models}, or even the existence of a \emph{unique} model in all cardinalities.  Another milestone result of Shelah in that direction is \cite{sh87a, sh87b}: 

\begin{fact}\label{shelah-multidim}
  Assume $2^{\aleph_n} < 2^{\aleph_{n + 1}}$ for all $n < \omega$. If an $\Ll_{\omega_1, \omega}$ sentence is categorical in every $\aleph_n$, then it is categorical in every infinite cardinal. 
\end{fact}

This has recently been generalized to AECs by Shelah and the author \cite{multidim-v2}. An example of Hart and Shelah \cite{hs-example, bk-hs} shows that one needs in general to assume categoricity at all $\aleph_n$'s to deduce categoricity further up. However, Mazari-Armida and the author \cite{mv-universal-jsl} showed that this restriction does not apply to ``simple'' AECs, and in particular to universal classes (classes of structures closed under isomorphisms, unions of chains, and substructures, see for example \cite{ap-universal-apal}). The reader may argue that there are relatively few interesting examples of universal classes, so let us work in the much broader (but slightly harder to define -- see Section \ref{prelim-sec} for the details) framework of tame AECs with intersections, encompassing multiuniversal classes \cite{abv-categ-multi-apal}, and Zilber's quasiminimal classes \cite{zil05}:

\begin{fact}\label{mv-fact}
  Assume $2^{\aleph_0} < 2^{\aleph_1}$. If $\K$ is an AEC with $\LS (\K) = \aleph_0$ which has intersections, is $\aleph_0$-tame, and is categorical in both $\aleph_0$ and $\aleph_1$, then it is categorical in all infinite cardinals.
\end{fact}
\begin{proof}
  By \cite[3.14]{logic-intersection-bpas}, $\K_{\le \aleph_0}$ is analytic, so the result follows from \cite[4.4]{mv-universal-jsl}.
\end{proof}

In the present paper, we aim to prove results along the ones above but \emph{in ZFC}. The broad idea is to develop a new, local, model theory that relies only on very weak consequences of the compactness theorem. In essence, using cardinal arithmetic makes it ``too easy'' by allowing set-theoretic tools (such as the weak diamond) to be used instead of model theory. While removing a minor-looking assumption such as ``$2^{\aleph_0} < 2^{\aleph_1}$'' may not seem very impressive, it is the author's thesis that in fact proving such results in ZFC is much harder and yields to interesting new mathematics.

One specific challenge is that it is harder to obtain amalgamation when one does not assume set-theoretic hypotheses: Shelah \cite[I.3.8]{shelahaecbook} has shown that categoricity in $\lambda$ and $\lambda^+$ implies amalgamation in $\lambda$ \emph{assuming} that $2^\lambda < 2^{\lambda^+}$. However, the set-theoretic hypothesis cannot in general be removed:

\begin{example}[{\cite[\S I.6]{shelahaecbook}}]
  Assuming Martin's axiom, there is an AEC (axiomatizable in $\Ll (\exists^{\ge \aleph_1})$) that is categorical in every cardinal in $[\aleph_0, 2^{\aleph_0})$, fails amalgamation everywhere below $2^{\aleph_0}$, and has no model of cardinality $\left(2^{\aleph_0}\right)^+$.
\end{example}

This example leads to the following weakening of Question \ref{shelah-q}, which is also open: 

\begin{question}\label{test-q-1}
  Assume $\K$ is an AEC with Löwenheim-Skolem-Tarski number $\lambda$, categorical in every cardinal in $[\lambda, 2^{\lambda}]$. Must $\K$ have a model of cardinality $\left(2^\lambda\right)^+$?
\end{question}

Similarly, looking at Fact \ref{shelah-multidim} suggests that to replace $\aleph_\omega$ by $\beth_\omega$ may be interesting:

\begin{question}\label{test-q-2}
  Assume $\K$ is an AEC with Löwenheim-Skolem-Tarski number $\lambda$, categorical in every cardinal in $[\lambda, \beth_\omega (\lambda))$. Must $\K$ be categorical everywhere above $\lambda$?
\end{question}

The present paper makes the following contributions:

\begin{enumerate}
\item By a very short proof, putting together several results of Shelah, we observe (Corollary \ref{basic-existence-cor}) that if $\mu > \LS (\K)$ is \emph{limit} and $\K$ is categorical everywhere in $[\LS (\K), \mu]$, then $\K$ has a model of cardinality $\mu^+$. This is a very partial approximation to Questions \ref{test-q-1} and \ref{test-q-2}, but perhaps it can awaken interest in these general cases again.
\item We give a positive answer to Question \ref{test-q-2} in the special case where $\K$ is a tame AEC with intersections (hence in particular when it is a universal class, see Corollary \ref{main-cor-2}). This gives a ZFC version of the Facts presented above, and partially answers \cite[Question 4.3]{mv-universal-jsl} (where we worked near $\aleph_1$ and assumed the weak continuum hypothesis, see Fact \ref{mv-fact}). We more generally answer Question \ref{test-q-2} in the broader framework of tame AECs with weak amalgamation and coherent sequences (see Section \ref{prelim-sec} for the definitions): \\

  \textbf{Corollary \ref{main-cor}.} Assume $\K$ is a $\LS (\K)$-tame AEC with weak amalgamation and coherent sequences. If $\K$ is categorical everywhere in $[\LS (\K), \beth_\omega (\LS (\K)))$, then $\K$ is categorical everywhere above $\LS (\K)$. \\
    
    Interestingly, even in the case of tame AECs with (full) amalgamation, the result is new and not trivial: it was only known \cite{tamenesstwo, tamenessthree} in case the AEC also has arbitrarily large models. 
\end{enumerate}

The proof of Corollary \ref{main-cor} proceeds as follows: we use tameness and the weak amount of amalgamation to build a good $\beth_\omega (\lambda)$-frame (we have set $\lambda := \LS (\K)$). The very rough idea is to follow the construction in \cite{ss-tame-jsl} (where amalgamation and arbitrarily large models were assumed), but there are many nontrivial difficulties because of the lack of amalgamation. We develop new, more local tools to get around this. One key is a local character theorem (Lemma \ref{splitting-lc}), generalizing \cite[4.12]{stab-spec-jml} which was instrumental in studying the stability spectrum for tame AECs. An important problem is how to build ``free'' (nonsplitting) extensions of types. The notion of a \emph{nicely fitrable model} (Definition \ref{nice-def}) is a new definition for a ``good'' saturated models in this amalgamationless context. In a sense, nicely filtrable models form the bases over which types behave well. They have independent interest and may play a role in future investigations. 

Once the good frame is built, a known upward transfer of good frames in tame AECs with weak amalgamation (see \cite{ext-frame-jml} and \cite[4.16]{ap-universal-apal}) is used to prove that the AEC has arbitrarily large models and eventual amalgamation. After some more work, it then becomes possible to use the result of Grossberg and VanDieren \cite{tamenessthree} showing that in tame AECs with amalgamation and arbitrarily large models, categoricity in two successive cardinals implies categoricity above those cardinals.

We assume the reader has some basic knowledge of AECs (see e.g.\ \cite{shelahaecbook, baldwinbook09, grossberg2002}), although we briefly repeat the main definitions in the preliminaries. In the last section, we will also assume some familiarity with the material in \cite{ss-tame-jsl} regarding good frames. Other results we use can be regarded as black boxes.

The author would like to thank John T.\ Baldwin for some interesting discussions (while on a research visit at UIC) that led to Section \ref{uic-sec} of the present paper. The author also thanks Marcos Mazari-Armida and a referee for helpful feedback on  this paper.

\section{Preliminaries}\label{prelim-sec}

\subsection{Abstract elementary classes}
Given a structure $M$, write $|M|$ for its universe and $\|M\|$ for the cardinality of its universe. We usually do not distinguish between $M$ and $|M|$, writing e.g.\ $a \in M$ instead of $a \in |M|$ and $A \subseteq M$ instead of $A \subseteq |M|$. An \emph{abstract class} is a pair $\K = (K, \lea)$, where $K$ is a class of structures in a fixed (here will all arities finite) vocabulary $\tau = \tau (\K)$ and $\lea$ is a partial order, $M \lea N$ implies that $M$ is a $\tau$-substructure of $N$, and both $K$ and $\lea$ respect isomorphisms (the definition is due to Grossberg). We often do not distinguish between $K$ (the class of structures) and $\K$ (the \emph{ordered} class of structures). Any abstract class admits a notion of \emph{$\K$-embedding}: these are the functions $f: M \rightarrow N$ such that $f: M \cong f[M]$ and $f[M] \lea N$. Thus one can naturally see $\K$ as a category. Unless explicitly stated, any map $f: M \rightarrow N$ in this paper will be a $\K$-embedding. We write $f: M \xrightarrow[A]{} N$ to mean that $f$ is a $\K$-embedding from $M$ into $N$ which fixes the set $A$ pointwise (so $A \subseteq M$). We similarly write $f: M \cong_A N$ for isomorphisms from $M$ onto $N$ fixing $A$.

For $\lambda$ a cardinal, we will write $\K_\lambda$ for the restriction of $\K$ to models of cardinality $\lambda$. Similarly define $\K_{\ge \lambda}$, $\K_{<\lambda}$, or more generally $\K_{\Theta}$, where $\Theta$ is a class of cardinals. 

For an abstract class $\K$, we denote by $\Ii (\K)$ the number of models in $\K$ up to isomorphism (i.e.\ the cardinality of $\K /_{\cong}$). We write $\Ii (\K, \lambda)$ instead of $\Ii (\K_\lambda)$. When $\Ii (\K) = 1$, we say that $\K$ is \emph{categorical}. We say that $\K$ is \emph{categorical in $\lambda$} if $\K_\lambda$ is categorical, i.e.\ $\Ii (\K, \lambda) = 1$.

We say that $\K$ has \emph{amalgamation} if for any $M_0 \lea M_\ell$, $\ell = 1,2$, there is $M_3 \in \K$ and $\K$-embeddings $f_\ell : M_\ell \xrightarrow[M_0]{} M_3$, $\ell = 1,2$. $\K$ has \emph{joint embedding} if any two models can be $\K$-embedded in a common model. $\K$ has \emph{no maximal models} if for any $M \in \K$ there exists $N \in \K$ with $M \lea N$ and $M \neq N$ (we write $M \lta N$). Localized concepts such as \emph{amalgamation in $\lambda$} mean that $\K_\lambda$ has amalgamation.

The definition of an abstract elementary class is due to Shelah \cite{sh88}:

\begin{defin}\label{aec-def}
  An \emph{abstract elementary class (AEC)} is an abstract class $\K$ satisfying:

  \begin{enumerate}
  \item Coherence: if $M_0, M_1, M_2 \in \K$, $M_0 \subseteq M_1 \lea M_2$ and $M_0 \lea M_2$, then $M_0 \lea M_1$.
  \item Tarski-Vaught chain axioms: if $\seq{M_i : i \in I}$ is a $\lea$-directed system and $M := \bigcup_{i \in I} M_i$, then:
    \begin{enumerate}
    \item $M \in \K$.
    \item $M_i \lea M$ for all $i \in I$.
    \item If $N \in \K$ is such that $M_i \lea N$ for all $i \in I$, then $M \lea N$.
    \end{enumerate}
  \item Löwenheim-Skolem-Tarski axiom: there exists a cardinal $\lambda \ge |\tau (\K)| + \aleph_0$ such that for any $N \in \K$ and any $A \subseteq N$, there exists $M \in \K$ with $M \lea N$, $A \subseteq M$, and $\|M\| \le |A| + \lambda$. We write $\LS (\K)$ for the least such $\lambda$.
  \end{enumerate}
\end{defin}

\subsection{Types}

In any abstract class $\K$, we can define a semantic notion of type (the definition was first given by Shelah in \cite{sh300-orig}). We give the full definition here for convenience, but the reader can also check \cite[2.16]{sv-infinitary-stability-afml} for more details. First, define a relation $E_{at}$ (atomic equivalence) on the class of triples $(b, A, N)$ with $N \in \K$, $A \subseteq N$, and $b \in N$ as follows: $(b_1, A, N_1) E_{at} (b_2, B, N_2)$ if $A = B$ and there exists $N \in \K$ and $\K$-embeddings $f: N_1 \xrightarrow[A]{} N$, $g: N_2 \xrightarrow[A]{} N$ so that $f (b_1) = g (b_2)$. Note that $E_{at}$ is a symmetric and reflexive relation. Let $E$ denote its transitive closure (if $\K$ has amalgamation, it is easy to check that $E_{at}$ is in fact already transitive, so $E = E_{at}$ in this case). For $N \in \K$, $A \subseteq N$, and $b \in N$, we define the \emph{type of $b$ over $A$ in $N$}, written $\gtp_{\K} (b / A; N)$, to be the $E$-equivalence class of the triple $(b, A, N)$. Usually $\K$ will be clear from context and we will omit it from the notation. These semantic types are called Galois (or orbital) types in the literature, but they coincide with the first-order syntactic types in elementary classes (and we will never use syntactic types anyway) so we simply call them ``types''. For $M \in \K$, we write $\gS_{\K} (M) = \gS (M)$ for $\{\gtp (b / M; N) \mid M \lea N\}$, the class\footnote{If $\K$ is an AEC, $\gS (M)$ will of course be a set.} of all types over $M$. Also define, for $M \lea N$, $\gS (M; N) = \{\gtp (b / M; N) \mid b \in N\}$, the class of all types over $M$ realized inside $N$. We define in the expected way what it means for a type to extend another type or to take the image of a type by a $\K$-embedding. We call a type $p$ \emph{algebraic} if it can be written as $p = \gtp (a / M; N)$, with $a \in M$.

As mentioned before, when $\K$ is an elementary class, $\gtp (b / A; N)$ contains the same information as the usual notion of $\Ll_{\omega, \omega}$-syntactic type. In particular, types in an elementary class are determined by their restrictions to finite sets. In \cite{tamenessone}, this fact was built into the following definition: for $\chi$ an infinite cardinal, an abstract class $\K$ is \emph{$(<\chi)$-tame} if for any $M \in \K$ and any distinct $p, q \in \gS (M)$, there exists $A \subseteq M$ such that $|A| < \chi$ and $p \rest A \neq q \rest A$. We say that $\K$ is \emph{$\chi$-tame} if it is $(<\chi^+)$-tame. Thus elementary classes are $(<\aleph_0)$-tame, but there are examples of non-tame AECs, see e.g.\ \cite[3.2.2]{bv-survey-bfo}.

\subsection{Weak amalgamation, intersections, and coherent sequences of types}

Weak amalgamation was first introduced in \cite[4.11]{ap-universal-apal}. It can be seen as a common weakening of amalgamation and having certain kinds of prime models.

\begin{defin}\label{weak-ap-def}
  Let $\K$ be an abstract class and let $M \in \K$. We say that $M$ is a \emph{weak amalgamation base} (in $\K$) if for any $N_1, N_2 \in \K$ with $M \lea N_1$, $M \lea N_2$, and any $a_1 \in N_1$, $a_2 \in N_2$, \emph{if} $\gtp (a_1 / M; N_1) = \gtp (a_2 / M; N_2)$, \emph{then} there exists $N_1^0, N_2', f$ such that:

  \begin{enumerate}
  \item $M \lea N_1^0 \lea N_1$.
  \item $a_1 \in N_1^0$.
  \item $N_2 \lea N_2'$.
  \item $f: N_1^0 \xrightarrow[M]{} N_2'$.
  \item $f (a_1) = a_2$.
  \end{enumerate}

  We say that $\K$ has \emph{weak amalgamation} if every object of $\K$ is a weak amalgamation base.
\end{defin}

Note that amalgamation implies weak amalgamation (see Fact \ref{type-ext-weak-ap} below). Another example of abstract classes with weak amalgamation are those that have intersections:

\begin{defin}
  An abstract class $\K$ \emph{has intersections} if for any $N \in \K$ and any $A \subseteq N$, the set $\bigcap \{M \in \K \mid M \lea N, A \subseteq M\}$ induces a $\K$-substructure of $N$. We write $\cl^N (A)$ for this substructure.
\end{defin}

\begin{remark}\label{tp-intersec}
  By \cite[1.3]{non-locality} or \cite[2.18]{ap-universal-apal}, in an abstract class with intersections, $\gtp (a_1 / M; N_1) = \gtp (a_2 / M; N_2)$ if and only if there exists an isomorphism $f: \cl^{N_1} (a_1 M) \cong_M \cl^{N_2} (a_2 M)$ such that $f (a_1) = a_2$. In particular, abstract classes with intersections have weak amalgamation.
\end{remark}

The following characterizes when weak amalgamation implies amalgamation.

\begin{fact}[{\cite[4.14]{ap-universal-apal}}]\label{type-ext-weak-ap}
  Let $\K$ be an AEC and let $\lambda \ge \LS (\K)$. The following are equivalent:

  \begin{enumerate}
  \item $\K_\lambda$ has weak amalgamation and for any $M \lea N$ both in $\K_\lambda$, any $p \in \gS (M)$ can be extended to a type in $\gS (N)$.
  \item $\K_\lambda$ has amalgamation.
  \end{enumerate}
\end{fact}

The definitions below are well known in AECs with amalgamation, see \cite[\S 11]{baldwinbook09}.

\begin{defin}\label{local-def}
  Let $\K$ be an abstract class, $\bar{M} = \seq{M_i : i < \alpha}$ be increasing continuous, and let $\bar{p} = \seq{p_i : i < \alpha}$ be an increasing chain of types with $p_i \in \gS (M_i)$ for all $i < \alpha$.

  \begin{enumerate}
  \item We say that $\bar{p}$ is \emph{local} if for any limit $i < \alpha$ and any $q \in \gS (M_i)$, if $q \rest M_j = p_j$ for all $j < i$, then $q = p_i$. We say that $\bar{M}$ is \emph{local} if any increasing chain of types $\bar{p}$ as above is local. 
  \item We say that $\bar{p}$ is \emph{coherent} if there exists $\seq{N_i : i < \alpha}$ increasing continuous, $\seq{a_i : i < \alpha}$, and $\seq{f_{ij} : i \le j < \alpha}$ such that for all $i \le j \le k < \alpha$:

    \begin{enumerate}
    \item $M_i \lea N_i$.
    \item $a_i \in N_i$.
    \item $p_i = \gtp (a_i / M_i; N_i)$.
    \item $f_{ij} : N_i \xrightarrow[M_i]{} N_{j}$.
    \item $f_{ii} = \id_{N_i}$.
    \item $f_{jk} \circ f_{ij} = f_{ik}$.
    \item $f_{ij} (a_i) = a_{j}$.
    \end{enumerate}

    We say that $\bar{M}$ is \emph{coherent} if any local $\bar{p}$ as above is coherent. Finally, we say that $\K$ \emph{has coherent sequences} if any $\bar{M}$ as above is coherent.
  \end{enumerate}
\end{defin}

The following is immediate from the definitions (take a directed colimit). See \cite[11.5]{baldwinbook09}.

\begin{fact}\label{local-constr}
  Let $\K$ be an AEC. Let $\delta$ be a limit ordinal, let $\bar{M} = \seq{M_i : i < \delta}$ be increasing continuous in $\K$, and let $\bar{p} = \seq{p_i : i < \delta}$ be an increasing chain of types with $p_i \in \gS (M_i)$ for all $i < \delta$. If $\bar{p}$ is coherent, then there exists $q \in \gS (\bigcup_{i < \delta} M_i)$ so that $q \rest M_i = q_i$ for all $i < \delta$.
\end{fact}

We finish by proving some easy results about building coherent sequences of types: it can be done assuming amalgamation or assuming the AEC has intersections. 

\begin{lem}\label{coherent-lem}
  Let $\K$ be an AEC. Let $\delta$ be a limit ordinal, let $\bar{M} = \seq{M_i : i < \delta}$ be increasing continuous in $\K$, and let $\bar{p} = \seq{p_i : i < \delta}$ be an increasing chain of types with $p_i \in \gS (M_i)$ for all $i < \delta$. Assume that $\bar{p}$ is local.

  \begin{enumerate}
  \item If $\K_{<\sup_{i < \delta}(\|M_i\| + \LS (\K))^+}$ has amalgamation, then $\bar{p}$ is coherent.
  \item If $\K$ has intersections, then $\bar{p}$ is coherent.
  \end{enumerate}

  In particular, AECs with amalgamation and AECs with intersections both have coherent sequences. 
\end{lem}
\begin{proof} \
  \begin{enumerate}
  \item Straightforward and well known. See the proof of \cite[11.5]{baldwinbook09}.
  \item We build the coherence witnesses $\seq{N_i : i < \alpha}$, $\seq{a_i : i < \alpha}$, $\seq{f_{ij} : i \le j < \alpha}$ inductively such that for all $i < \alpha$, $N_i = \cl^{N_i} (M_i a_i)$. The base case is trivial, and at limits we can take directed colimits (and use locality to see type equality is preserved; the closure condition will be preserved by finite character of the closure operator \cite[2.14(6)]{ap-universal-apal}). At successors, we are given $N_i$ and $a_i$ and want to build $N_{i + 1}$, $a_{i + 1}$, and $f_{i(i + 1)}$. Pick $N_{i + 1}'$ and $a_{i + 1}$ such that $p_{i + 1} = \gtp (a_{i + 1} / M_{i + 1}; N_{i + 1}')$. Then $p_{i + 1} \rest M_i = p_i$, so there exists an isomorphism $f_{i} : N_i \cong \cl^{N_{i + 1}'} (M_i a_{i + 1})$ such that $f (a_i) = a_{i + 1}$. Let $N_{i + 1} := \cl^{N_{i + 1}'} (M_{i + 1} a_{i + 1})$. By the coherence axiom of AECs, $\cl^{N_{i + 1}'} (M_i a_{i + 1}) \lea N_{i + 1}$. Thus composing $f_i$ with the inclusion gives the desired $\K$-embedding $f_{i(i + 1)}$ of $N_i$ into $N_{i + 1}$.
  \end{enumerate}
\end{proof}

\section{Existence from successive categoricity below a limit}\label{uic-sec}

In this section, we work in an arbitrary AEC (i.e.\ we do not assume tameness or weak amalgamation) and prove some relatively easy results about deriving no maximal models from categoricity below a limit.

The following easy observation will be used in the later sections:

\begin{lem}\label{limit-existence}
  Let $\K$ be an AEC and let $\mu > \LS (\K)$ be a limit cardinal of countable cofinality. If $\K$ is categorical in unboundedly-many cardinals below $\mu$, then $\K_\mu \neq \emptyset$.
\end{lem}
\begin{proof}
  Since $\mu$ has countable cofinality, we can pick $\seq{\mu_i : i < \omega}$ increasing cofinal in $\mu$ such that $\mu_0 \ge \LS (\K)$ and $\mu_i$ is a categoricity cardinal for each $i < \omega$. Now we build an increasing chain $\seq{M_i : i < \omega}$ such that $M_i \in \K_{\mu_i}$ for all $i < \omega$. Then the union of the chain will be in $\K_\mu$. For $i = 0$, take any $M_0 \in \K_{\mu_0}$. For $i = j + 1$, given $M_j$, first pick any $N_i \in \K_{\mu_i}$. Pick $N_i^0 \lea N_i$ with $N_i^0 \in \K_{\mu_j}$. By categoricity, there is an isomorphism $f: N_i^0 \cong M_j$. With some renaming, we can extend this isomorphism to $g: N_i \cong M_i$, for some $M_i$ with $M_j \lea M_i$. Then $M_i \in \K_{\mu_i}$, as desired.
\end{proof}

The next two definitions are due to Shelah \cite{sh576}.

\begin{defin}
  An \emph{existence triple} in an abstract class $\K$ is a triple $(a, M, N)$ with $M \lea N$ both in $\K$ and $a \in N \backslash M$.
\end{defin}

\begin{defin}
  We say an abstract class $\K$ has \emph{weak extension} if for any existence triple in $(a, M, N)$ in $\K$, there exists an existence triple $(b, M', N')$ in $\K$ with $a = b$, $M \lta M'$, and $N \lta N'$. We say that $(b, M', N')$ is a \emph{strict extension} of $(a, M, N)$.
\end{defin}

The next two results are essentially due to Shelah, see \cite[\S VI.1]{shelahaecbook}. We give full proofs here because they are short and (for the second one) slightly simpler than Shelah's.

\begin{lem}\label{ext-1}
  Let $\K$ be an AEC and let $\mu > \LS (\K)$. If:
  
  \begin{enumerate}
  \item There exists an existence triple in $\K_{\LS (\K)}$.
  \item For every $\lambda \in [\LS (\K), \mu)$, $\K_\lambda$ has weak extension.
  \end{enumerate}

  Then not every element of $\K_\mu$ is maximal.
\end{lem}
\begin{proof}
  Pick an existence triple $(a, M, N)$ in $\K_{\LS (\K)}$. Now build $\seq{M_i : i \le \mu}$, $\seq{N_i : i \le \mu}$ both increasing continuous such that for every $i < \mu$:

  \begin{enumerate}
  \item $M_0 = M$, $N_0 = N$.
  \item $M_i, N_i \in \K_{|i| + \LS (\K)}$.
  \item $M_i \lea N_i$.
  \item $a \in N_i \backslash M_i$.
  \item $M_{i} \lta M_{i + 1}$.
  \end{enumerate}

  This is possible using the weak extension property at successor stages and taking unions at limit stages. In the end, $M_\mu \lta N_\mu$, since $a \in N_\mu \backslash M_\mu$. Thus $M_\mu$ is not maximal. Since $\seq{M_i : i \le \mu}$ was strictly increasing,  $M_\mu \in \K_\mu$.
\end{proof}

\begin{lem}\label{ext-2}
  Let $\K$ be an AEC. If $\K$ is categorical in $\LS (\K)$ and $\K_{\LS (\K)^+}$ has no maximal models, then $\K_{\LS (\K)}$ has weak extension.
\end{lem}
\begin{proof}
  Let $\lambda := \LS (\K)$. Let $(a, M, N)$ be an existence triple in $\K_{\lambda}$. We want to find a strict extension of $(a, M, N)$. We build $\seq{M_i : i \le \lambda^+}$ increasing continuous and $\seq{a_i : i < \lambda^+}$ such that for all $i < \lambda^+$:

  \begin{enumerate}
  \item $(a_i, M_i, M_{i + 1})$ is an existence triple in $\K_{\lambda}$.
  \item There exists an isomorphism $f_i : N \cong M_{i + 1}$ so that $f_i[M] = M_i$ and $f_i (a) = a_i$.
  \end{enumerate}

  This is possible by categoricity. This is enough: since $\K_{\lambda^+}$ has no maximal models, there exists $M' \in \K_{\lambda^+}$ with $M_{\lambda^+} \lta M'$. Let $\seq{M_i' : i \le \lambda^+}$ be an increasing continuous resolution of $M'$. Let $C \subseteq \lambda^+$ be a club such that $i \in C$ implies $M_i \lea M_i'$ and moreover $M_{\lambda^+} \cap M_i' = M_i$. Pick $b \in M' \backslash M_{\lambda^+}$ and pick $i \in C$ so that $b \in M_i'$. Then $M_i \lta M_i'$ and one can pick $j < \lambda^+$ so that $M_{i + 1} \lta M_j'$. Then $(a_i, M_i', M_j')$ is a strict extension of $(a_i, M_i, M_{i + 1})$. Taking an isomorphic copy, we obtain a strict extension of the original triple $(a, M, N)$.
\end{proof}

Putting the two lemmas together, we obtain:

\begin{thm}\label{basic-existence-thm}
  Let $\K$ be an AEC and let $\mu > \LS (\K)$ be a limit cardinal. If $\K$ is categorical in every cardinal in $[\LS (\K), \mu)$, then not every element of $\K_\mu$ is maximal.
\end{thm}
\begin{proof}
  Note that for every $\lambda \in [\LS (\K), \mu)$, both $\K_{\lambda}$ and $\K_{\lambda^+}$ are not empty and have no maximal models (we are using that $\mu$ is limit to deduce this for $\K_{\lambda^+}$). In particular, there is an existence triple in $\K_{\LS (\K)}$. Moreover by Lemma \ref{ext-2} (applied to $\K_{\ge \lambda}$), for any $\lambda \in [\LS (\K), \mu)$, $\K_\lambda$ has the weak extension property. By Lemma \ref{ext-1}, we get the result.
\end{proof}
\begin{cor}\label{basic-existence-cor}
  Let $\K$ be an AEC and let $\mu > \LS (\K)$ be a limit cardinal. If $\K$ is categorical in every cardinal in $[\LS (\K), \mu]$, then $\K_{\mu^+} \neq \emptyset$.
\end{cor}
\begin{proof}
  By Theorem \ref{basic-existence-thm}, not every model in $\K_\mu$ is maximal. By categoricity in $\mu$, this implies that $\K_\mu$ has no maximal models, hence that $\K_{\mu^+} \neq \emptyset$.
\end{proof}

\section{Local saturation and splitting}

The present section is the core of the paper. We adapt known results about saturated models and splitting to setups without amalgamation (but often with weak amalgamation and/or tameness). For several results, categoricity is not needed, but in the end we will use it to put everything together.

The following are localized definitions of the well known variations of saturation:

\begin{defin}\label{sat-def}
  Let $\K$ be an AEC. For $\theta$ an infinite cardinal, a model $M$ is \emph{locally $\theta$-saturated} if for any $N \gea M$ and any $A \subseteq M$ with $|A| < \theta$, we have that $\gS (A; N) = \gS (A; M)$. For $M_0 \lea M$, we say that $M$ is \emph{locally $\theta$-universal over $M_0$} if for any $M_0 \lea N_0 \lea N$ with $M \lea N$ and $\|N_0\| < \theta$, $N_0$ embeds into $M$ over $M_0$. $M$ is \emph{locally $\theta$-model-homogeneous} if it is $\theta$-universal over $M_0$ for any $M_0 \lea M$ with $\|M_0\| < \theta$. We say that $M$ is \emph{locally saturated} when it is locally $\|M\|$-saturated, and similarly for locally model-homogeneous. We say that $M$ is \emph{locally universal over $M_0$} if it is locally $\|M_0\|^+$-universal over $M_0$.
\end{defin}

The usual exhaustion argument shows that locally model-homogeneous model can be built assuming some cardinal arithmetic. See for example the proof of \cite[Theorem 1]{rosicky-sat-jsl}.

\begin{fact}\label{mh-dense}
  Let $\K$ be an AEC. For any $M \in \K$ and any regular cardinal $\theta > \LS (\K)$, there exists $N \in \K$ such that $M \lea N$, $N$ is locally $\theta$-model-homogeneous, and $\|N\| \le \|M\|^{<\theta}$.
\end{fact}

Assuming categoricity in a suitable unbounded set, we can build locally model-homogeneous models. 

\begin{lem}\label{categ-mh}
  Let $\K$ be an AEC and let $\mu > \LS (\K)$ be a strong limit cardinal. If for every $\theta < \mu$ there exists $\lambda < \mu$ such that $\lambda = \lambda^\theta$ and $\K$ is categorical in $\lambda$, then every object of $\K_\mu$ is locally model-homogeneous.
\end{lem}
\begin{proof}
  Let $M \in \K_\mu$. Fix $\theta < \mu$, $N \in \K_\mu$ with $M \lea N$, and $M_0, N_0 \in \K_{\le \theta}$ with $M_0 \lea M$, $M_0 \lea N_0 \lea N$. We have to see that $N_0$ embeds inside $M$ over $M_0$. Without loss of generality, $\theta \ge \LS (\K)$. By assumption, we can pick a categoricity cardinal $\lambda < \mu$ such that $\lambda = \lambda^\theta$. Let $M_0' \lea M$ be such that $M_0 \lea M_0'$ and $M_0' \in \K_{\lambda}$. By categoricity in $\lambda$ and Fact \ref{mh-dense}, $M_0'$ is locally $\theta^+$-model-homogeneous. In particular, $N_0$ embeds into $M_0'$ over $M_0$. Thus $N_0$ embeds into $M$ over $M_0$, as desired.
\end{proof}

The following notion was introduced by Shelah \cite[3.2]{sh394} for AECs with amalgamation. 

\begin{defin}[Splitting]
  Let $\K$ be an AEC, let $N \in \K$, $A \subseteq N$, let $p \in \gS (N)$, and let $\theta$ be an infinite cardinal with $|A| < \theta$. We say that $p$ \emph{$(<\theta)$-splits over $A$} if there exists $M_1, M_2 \in \K_{<\theta}$ such that $A \subseteq M_\ell \lea N$ for $\ell = 1,2$ and $f: M_1 \cong_A M_2$ so that $f (p \rest M_1) \neq p \rest M_2$. We say that $p$ \emph{$\theta$-splits over $A$} if it $(<\theta^+)$-splits over $A$. We say that $p$ \emph{splits over $A$} if it $(|A| + \aleph_0)$-splits over $A$.
\end{defin}

\begin{remark}\label{spl-witness-rmk}
  Let $N \in \K$, $A \subseteq N$, let $p \in \gS (N)$ and let $\theta$ be an infinite cardinal with $|A| < \theta$. If $p$ $(<\theta)$-splits over $A$, then there exists $N_0 \lea N$ with $A \subseteq N_0$ and $\|N_0\| < \theta$ such that $p \rest N_0$ $(<\theta)$-splits over $A$.

\end{remark}

The following is a generalization of \cite[4.12]{stab-spec-jml}.

\begin{lem}[Local character]\label{splitting-lc}
  Let $\K$ be an AEC, let $M \in \K_{\ge \aleph_0}$, let $\delta$ be a regular cardinal, and let $\seq{A_i : i \le \delta}$ be an increasing continuous chain of sets with $A_\delta = M$. Let $\theta$ be either $\|M\|^+$ or $\sup_{i < \delta} |A_i|^+$. Let $p \in \gS (M)$. If there exists a regular $\chi \le \|M\|$ such that:

  \begin{enumerate}
  \item $\K_{<\theta}$ is $(<\chi)$-tame.
  \item $M$ is locally $(\chi + \delta^+)$-saturated.
  \end{enumerate}

  Then there exists $i < \delta$ such that $p$ does not $(<\theta)$-split over $A_i$.
\end{lem}
\begin{proof}
  Suppose not. Then for each $i < \delta$, there exists $M_i^1, M_i^2 \in \K_{<\theta }$ and $f_i : M_i^1 \cong_{A_i} M_i^2$ so that $A_i \subseteq M_i^\ell \lea M$, $\ell = 1,2$ and $f_i (p \rest M_i^1) \neq p \rest M_i^2$. 

  By $(<\chi)$-tameness, we can find $A_i^1 \subseteq M_i^1$ of cardinality strictly less than $\chi$ so that $f_i (p \rest A_i^1) \neq p \rest A_i^2$ (we have set $A_i^2 := f_i[A_i^1])$.

  Let $A := \bigcup_{i < \delta} (A_i^1 \cup A_i^2)$. Recall that $p$ is realized in an extension of $M$. Moreover, if $\delta < \chi$ then $|A| < \chi$ and if $\delta \ge \chi$ then $|A| \le \delta$. In either case, $M$ is locally $|A|^+$-saturated, so $p \rest A$ is realized in $M$, say by $a$. Since $\delta$ is a limit ordinal, there exists $i < \delta$ such that $a \in A_i$. Now, fix an extension $g_i : M \cong M'$ of $f_i$ (so $M_i^2 \lea M'$). We have:

  $$
  f_i (p \rest A_i^1) = g_i (\gtp (a / A_i^1; M)) = \gtp (g_i (a) / A_i^2; M') = \gtp (a / A_i^2; M')
  $$

  Where we have used that $g_i$ fixes $A_i$, so $g_i (a) = a$. On the other hand, since $A_i^2 \subseteq A$,

  $$
  p \rest A_i^2 = \gtp (a / A_i^2; M)
  $$

  Finally, observe that $M_i^2 \lea M$ and $M_i^2 \lea M'$ by construction. Therefore since $a \in A_i \subseteq M_i^2$, $\gtp (a / A_i^2; M') = \gtp (a / A_i^2; M_i^2) = \gtp (a / A_i^2; M)$. We have shown that $f_i (p \rest A_i^1) = p \rest A_i^2$, contradicting the definition of $A_i^1, A_i^2$, and $f_i$.
\end{proof}

We now generalize the weak uniqueness and extension properties of splitting first isolated by VanDieren \cite[I.4.10]{vandierennomax}.

\begin{lem}[Weak uniqueness]\label{weak-uq}
  Let $M_0 \lea M_1 \lea N$ all be in $\K_{\ge \LS (\K)}$. Let $p, q \in \gS (N)$. If $M_1$ is locally universal over $M_0$, $p$ and $q$ both do not split over $M_0$, and $p \rest M_1 = q \rest M_1$, then $p \rest A = q \rest A$ for any $A \subseteq N$ with $|A| \le \|M_0\|$.
\end{lem}
\begin{proof}
  Fix $A \subseteq N$ with $|A| \le \|M_0\|$. Pick $N_0 \lea N$ with $A \subseteq N_0$ and $\|N_0\| \le \|M_0\|$. Since $M_1$ is locally universal over $M_0$, there exists $f: N_0 \xrightarrow[M_0]{} M_1$. By the definition of nonsplitting (where $M_1, M_2$ there stand for $N_0$, $f[N_0]$ here), we must have that $f (p \rest N_0) = p \rest f[N_0]$ and $f (q \rest N_0) = q \rest f[N_0]$. Since $p \rest M_1 = q \rest M_1$, $f \rest f[N_0] = q \rest f[N_0]$. Therefore $f (p \rest N_0) = f (q \rest N_0)$, hence $p \rest N_0 = q \rest N_0$.
\end{proof}

\begin{lem}[Weak extension]\label{weak-ext}
  Let $M_0 \lea M_1 \lea M$ all be in $\K_{\ge \LS (\K)}$ and let $p \in \gS (M)$. Let $N_0 \in \K$ be such that $M_1 \lea N_0$. If:

  \begin{enumerate}
  \item $p$ does not split over $M_0$.
  \item $N_0$ embeds into $M$ over $M_1$.
  \end{enumerate}
  
  Then there exists $q_0 \in \gS (N_0)$ such that $q_0$ extends $p \rest M_1$ and $q_0$ does not split over $M_0$. Moreover, if $p$ is not algebraic then $q_0$ can be taken to be nonalgebraic.
\end{lem}
\begin{proof}
  Use the hypothesis to fix $f: N_0 \xrightarrow[M_1]{} M$. Let $q_0 := f^{-1} (p \rest f[N_0])$. By invariance, $q_0$ does not split over $M_0$ and $q_0 \rest M_1 = p \rest M_1$. To see the moreover part, assume that $q_0$ is algebraic. Then $p \rest f[N_0]$ is algebraic, hence $p$ is algebraic.
\end{proof}

We now work toward improving the statement of weak extension. Roughly, we would like to prove that if $M$ is model-homogeneous and $N$ is an extension of $M$ of the same cardinality, then types over $M$ have nonsplitting extensions over $N$. This is not immediate from Lemma \ref{weak-ext} because without amalgamation we do not know whether $N$ embeds into $M$. Instead, we will first build small approximations of the type we want to build, and then use tameness and existence of coherent sequences of types to take the limit of these approximations. The following technical definition will be key (one can think of it as a replacement for the notion of a ``good saturated model'' -- or limit model -- in this context, see Remark \ref{nice-rmk}):

\begin{defin}\label{nice-def}
  Let $\K$ be an AEC. We call $N \in \K$ \emph{nicely filtrable over $M$} if there exists a limit ordinal $\delta$ and an increasing continuous chain $\bar{N} = \seq{N_i : i < \delta}$ such that:

  \begin{enumerate}
  \item $N = \bigcup_{i < \delta} N_i$, $N_0 = M$.
  \item $\|N_i\| < \|N\|$ for all $i < \delta$.
  \item $N_{i + 1}$ is locally universal over $N_i$ for all $i < \delta$.
  \item $\bar{N}$ is local and coherent (Definition \ref{local-def}).
  \item For any $p \in \gS (\bigcup_{i < \delta} N_i)$, there exists $i < \delta$ so that $p$ does not split over $N_i$.
  \end{enumerate}

  We call $\seq{N_i : i < \delta}$ a \emph{nice filtration of $N$ (over $M$)}.
\end{defin}
\begin{remark}\label{nice-rmk}
  It is not clear whether in general if $N$ is nicely filtrable over $M$, then $N$ is nicely filtrable over any $M'$ with $\|M'\| < \|N\|$, although this seems to happen in reasonable cases. For example, if $\K$ is the class of models of a superstable first-order theory $T$ (ordered by elementary substructure), then a model $N$ is nicely filtrable (over some base) if and only if it is saturated. This holds even if $T$ is only stable with $\cf{\|N\|} \ge \kappa (T)$, and generalizes (in high-enough cardinals) to tame AECs with amalgamation, see for example \cite[\S6]{stab-spec-jml}.
\end{remark}

From categoricity in a suitable unbounded set of cardinals below a strong limit of countable cofinality $\mu$, we can get nice filtrations. Note that Lemma \ref{limit-existence} tells us that from the hypotheses of Lemma \ref{nice-resolv-lem} there will be a model in $\K_\mu$.

\begin{lem}\label{nice-resolv-lem}
  Let $\K$ be an AEC and let $\mu > \LS (\K)$ be a strong limit cardinal of countable cofinality. If:

  \begin{enumerate}
  \item For any $\theta < \mu$ there exists $\lambda < \mu$ such that $\lambda = \lambda^\theta$ and $\K$ is categorical in $\lambda$.
  \item $\K_{<\mu}$ is $\LS (\K)$-tame.
  \item $\K_{<\mu}$ has coherent sequences.
  \end{enumerate}

  Then for any $N \in \K_{\mu}$ and any $M \in \K_{[\LS (\K),\mu)}$ with $M \lea N$, $N$ is nicely filtrable over $M$.
\end{lem}
\begin{proof}
  Let $N \in \K_{\mu}$ and let $M \in \K_{[\LS (\K), \mu)}$ with $M \lea N$. Pick an increasing sequence $\seq{\mu_i : i < \omega}$ cofinal in $\mu$ with $\|M\| \le \mu_0$ and so that for any $i < \omega$, $\mu_{i + 1}^{\mu_i} = \mu_{i + 1}$ and $\K$ is categorical in $\mu_i$. This is possible by assumption. By increasing $M$ if needed, we can assume without loss of generality that $M \in \K_{\mu_0}$. Now pick any increasing sequence $\seq{N_i : i < \omega}$ such that $N_0 = M$, $\bigcup_{i < \omega} N_i = N$, and $N_i \in \K_{\mu_i}$ for all $i < \omega$. We claim this is a nice filtration of $N$ over $M$.

    As in the proof of Lemma \ref{categ-mh}, for any $i < \omega$, $N_{i + 1}$ is locally universal over $N_i$. Also, $\bar{N} = \seq{N_i : i < \omega}$ is trivially local since there are no limit ordinals below $\omega$. Furthermore, $\bar{N}$ is coherent because by assumption $\K_{<\mu}$ has coherent sequences. Finally, pick $p \in \gS (N)$. Note that by Lemma \ref{categ-mh}, $N$ is locally model-homogeneous, hence locally saturated. Thus Lemma \ref{splitting-lc} (where $\delta, \chi$ there stands for $\omega, \LS (\K)^+$ here) implies there exists $i < \omega$ so that $p$ does not split over $N_i$. 
\end{proof}

From the existence of nice filtrations, we can now prove the desired extension property:

\begin{lem}[Extension]\label{ext-lem}
  Let $\K$ be an AEC and let $M_0 \lea M_1 \lea M \lea N$ all be in $\K_{\ge \LS (\K)}$. Let $p \in \gS (M)$. If:

  \begin{enumerate}
  \item $p$ does not split over $M_0$.
  \item $M_1$ is locally universal over $M_0$.
  \item $M$ is $(\|M\|)$-locally universal over $M_1$.
  \item $\|N\| = \|M\|$.
  \item $N$ is nicely filtrable over $M_1$.
  \item $\K_{\le \|N\|}$ is $\|M_0\|$-tame.
  \end{enumerate}

  Then there exists $q \in \gS (N)$ which extends $p$ and does not split over $M_0$. Moreover, $q$ is not algebraic if $p$ is not algebraic.
\end{lem}
\begin{proof}
  Fix $\bar{N} = \seq{N_i : i < \delta}$ a nice filtration of $N$ over $M_1$. For $i < \delta$, let $q_i \in \gS (N_i)$ be as given by weak extension: it extends $p \rest M_1$ and does not split over $M_0$. Moreover it is nonalgebraic if $p$ is nonalgebraic. By weak uniqueness (and tameness), $q_j \rest N_i = q_i$ for $i < j < \delta$. By Fact \ref{local-constr}, there exists $q \in \gS (N)$ such that $q \rest N_i = q_i$ for all $i < \delta$. Note that $q$ cannot be algebraic if all the $q_i$'s are nonalgebraic. By the properties of $\bar{N}$, there exists $i < \delta$ so that $q$ does not split over $N_i$.

  \underline{Claim}: $q$ does not split over $M_0$.

  \underline{Proof of Claim:} By Remark \ref{spl-witness-rmk}, it suffices to see that $q \rest N_0'$ does not split over $M_0$ for any $N_0' \in \K_{<\|N\|}$ with $N_{i + 1} \lea N_0' \lea N$. So let $q_0' \in \gS (N_0')$ be as given by weak extension (extending $p \rest M_1$ and not splitting over $M_0$). By weak uniqueness, $q_0' \rest N_{i + 1} = q_{i + 1}$. Since $q$ does not split over $N_i$, $q \rest N_0'$ does not split over $N_i$. By monotonicity, also $q_0'$ does not split over $N_i$. By weak uniqueness again (recalling that by definition $N_{i + 1}$ is locally universal over $N_i$), $q_0' = q \rest N_0'$. In particular, $q \rest N_0'$ does not split over $M_0$. $\dagger_{\text{Claim}}$ 

  Now since $q \rest M_1 = p \rest M_1$, weak uniqueness and tameness imply that $q \rest M = p$, as desired.
\end{proof}

\section{Good frames and the main theorem}

We use the tools of the previous section to build a good frame, a local forking-like notion. We assume some familiarity with good frames (see \cite[Chapter 2]{shelahaecbook}) here. We will use the definition and notation from \cite[\S 2.4]{ss-tame-jsl} (which we do not repeat here). Recall that a \emph{$\goodms{S}$} $\lambda$-frame is a good frame, except that it may fail the symmetry property. The following key result tells us that good frame can be transferred up in tame AECs with weak amalgamation:

\begin{fact}\label{frame-upward}
  Let $\s$ be a $\goodms{S}$ $\lambda$-frame on an AEC $\K$. If $\K$ is $\lambda$-tame and has weak amalgamation, then $\s$ is a good $\lambda$-frame and extends to a good $[\lambda, \infty)$-frame on all of $\K_{\ge \lambda}$. In particular, $\K_{\ge \lambda}$ has amalgamation and arbitrarily large models.
\end{fact}
\begin{proof}
  By \cite[4.16]{ap-universal-apal}, $\K_{\ge \lambda}$ has amalgamation. By \cite{ext-frame-jml}, $\s$ extends to a $\goodms{S}$ $[\lambda, \infty)$-frame $\ts$ on $\K$. By \cite[6.14]{ss-tame-jsl}, $\ts$, and hence $\s$, also has symmetry\footnote{This will not be used in the present paper.}. The ``in particular'' part follows from the definition of a good $[\lambda, \infty)$-frame.
\end{proof}

For $\K$ an AEC and $M \in \K$, we let $\K_M$ denote the AEC obtained by adding constant symbols for $M$ (see \cite[2.20]{ap-universal-apal} for the precise definition). Roughly speaking, it is the AEC of models above $M$. 

\begin{lem}[Main lemma]\label{main-lem-frame}
  Let $\K$ be an AEC and let $\mu > \LS (\K)$ be a strong limit cardinal of countable cofinality. \emph{If}:
  \begin{enumerate}
  \item $\K_{\mu}$ has weak amalgamation.
  \item $\K_{<\mu}$ has coherent sequences.
  \item $\K_{\le \mu}$ is $\LS (\K)$-tame.
  \item For every $\theta < \mu$, there exists $\lambda < \mu$ such that $\lambda = \lambda^\theta$ and $\K$ is categorical in $\lambda$.
  \end{enumerate}

  \emph{Then} for any non-maximal $M \in \K_\mu$, $\K_M$ has a type-full $\goodms{S}$ $\mu$-frame.
\end{lem}
\begin{proof}
  For $N_1 \lea N_2$ both in $\K_\mu$ and $p \in \gS (N_2)$, we say that $p$ \emph{does not fork over $N_1$} if there exists $N_1^0 \in \K_{<\mu}$ with $N_1^0 \lea N_1$ such that $p$ does not split over $N_1^0$. We prove several claims:

  \begin{itemize}
  \item Nonforking is invariant under isomorphisms, and if $N_1 \lea N_2 \lea N_3$ are all in $\K_\mu$ and $p \in \gS (N_3)$ does not fork over $N_1$, then $p \rest N_2$ does not fork over $N_1$ and $p$ does not fork over $N_2$. This is immediate from the definition and the basic properties of splitting.
  \item If $N_1 \lea N_2$ are both in $\K_\mu$, $p, q \in \gS (N_2)$ both do not fork over $N_1$ and $p \rest N_1 = q \rest N_1$, then $p = q$. To see this use the Löwenheim-Skolem axiom, fix $N_1^0 \lea N_1$ such that $N_1^0 \in \K_{<\mu}$ and both $p$ and $q$ do not split over $N_1^0$. By Lemma \ref{categ-mh}, all the models in $\K_\mu$ are locally model-homogeneous, so in particular $N_1$ is locally universal over $N_1^0$, so by tameness and weak uniqueness (Lemma \ref{weak-uq}), $p = q$.
  \item If $N \in \K_\mu$ and $p \in \gS (N)$, then $p$ does not fork over $N$. This follows directly from the definition of nice filtrations and Lemma \ref{nice-resolv-lem}.
  \item If $N_1 \lea N_2$ are both in $\K_\mu$ and $p \in \gS (N_1)$, then there exists $q \in \gS (N_2)$ such that $q$ extends $p$ and $q$ does not fork over $N_1$. Moreover, $q$ can be taken to be nonalgebraic if $p$ is nonalgebraic. To see this, first note that we have observed previously that $p$ does not fork over $N_1$, hence there is $M_0 \in \K_{[\LS (\K), \mu)}$ such that $M_0 \lea N_1$ and $p$ does not split over $M_0$. Now we can pick $M_1 \lea N_1$ with $M_0 \lea M_1$, $\|M_0\| < \|M_1\| < \|N_1\|$, and $\|M_1\|^{\|M_0\|} = \|M_1\|$, and $\K$ is categorical in $\|M_1\|$. Then it follows that $M_1$ is locally universal over $M_0$ (see Fact \ref{mh-dense}). Also, Lemma \ref{nice-resolv-lem} implies that $N_1$ is nicely filtrable over $M_1$. Now apply Lemma \ref{ext-lem}, where $M, N$ there stand for $N_1, N_2$ here.
  \item If $\delta < \mu^+$ is a limit ordinal, $\seq{N_i : i \le \delta}$ is increasing continuous in $\K_\mu$, and $p \in \gS (N_\delta)$, then there exists $i < \delta$ such that $p$ does not fork over $N_i$. Indeed, suppose not. Assume without loss of generality that $\delta$ is regular. In particular, $\delta < \mu$ and so $\delta^+ < \mu$. Let $\theta := \LS (\K)^+ + \delta^+$. Pick $\lambda < \mu$ such that $\K$ is categorical in $\lambda$ and $\lambda^\theta = \lambda$. By Fact \ref{mh-dense}, the model in $\K_\lambda$ is locally $\theta^+$-model-homogeneous. Build $\seq{N_i^0 : i < \delta}$ and $\seq{N_i^1 : i < \delta}$ increasing in $\K_{\lambda}$ such that for all $i < \delta$, $N_i^0 \lea N_i$, $N_i^0 \lea N_i^1 \lea N_\delta$, $N_{i}^1 \cap N_i \subseteq N_{i + 1}^0$,  $p \rest N_i^1$ splits over $N_i^0$. This is possible (see \cite[4.11]{ss-tame-jsl} for a very similar construction). At the end, let $N_\delta^\ell := \bigcup_{i < \delta} N_i^\ell$. Observe that $N_\delta^1 = N_\delta^0$ so by Lemma \ref{splitting-lc}, there exists $i < \delta$ so that $p \rest N_\delta^1$ does not split over $N_i^0$. This is a contradiction, since we assumed that $p \rest N_i^1$ splits over $N_i^0$.
  \item For any $N \in \K_\mu$, $|\gS (N)| \le \mu$. Indeed, if $\seq{p_i : i < \mu^+}$ are types in $\gS (N)$, we can first use Lemma \ref{nice-resolv-lem} to find a nice filtration $\seq{N_j : j < \omega}$ of $N$. In particular, for all $i < \mu^+$ there exists $j_i < \omega$ with $p$ not splitting over $N_{j_i}$. By the usual pruning argument (using that $\mu$ is strong limit to see that $|\gS (N_j)| < \mu$ for each $j < \mu$), there exists $j < \omega$ and an unbounded subset $X \subseteq \mu^+$ such that for $i_1, i_2 \in X$, $p_{i_1}$, $p_{i_2}$ both do not split over $N_j$ and have the same restriction to $N_{j + 1}$. By weak uniqueness, $p_{i_1} = p_{i_2}$. This shows that $|\gS (N)| \le \mu$.
  \end{itemize}

  Now fix a non-maximal $M \in \K_\mu$. We identify any $N \in \K_\mu$ such that $M \lea N$ with the corresponding member of $\K_M$ (this will not yield to confusion). We define nonforking in the frame using the nonforking relation defined above. The basic types are the nonalgebraic types. We have just established that invariance, monotonicity, uniqueness, extension, and local character hold in the frame. Therefore transitivity and continuity automatically follow (see \cite[II.2.17,II.2.18]{shelahaecbook}). Also note that $\K_M$ is not empty (it contains a copy of $M$) and has no maximal models of cardinality $\mu$: given $M \lea N$ in $\K_\mu$, we can fix a nonalgebraic $p \in \gS (M)$ (which exists as $M$ is not maximal) and take its nonforking extension to $\gS (N)$. We have observed this extension is not algebraic, hence $N$ cannot be maximal either. We have also seen that $\K$ (hence $\K_M$) is stable in $\mu$. Finally, $\K$ has amalgamation in $\mu$. This follows from Fact \ref{type-ext-weak-ap} and the extension property of nonforking. We deduce that $\K_M$ has amalgamation in $\mu$ and also joint embedding in $\mu$.
\end{proof}
\begin{remark}
  By Lemma \ref{limit-existence}, any AEC satisfying the hypotheses of Lemma \ref{main-lem-frame} will have a model in $\K_\mu$, but we do not in general know how to find a non-maximal one. Theorem \ref{basic-existence-thm} gives a way by assuming categoricity in more cardinals. 
\end{remark}

We deduce arbitrarily large models from categoricity in enough small cardinals:

\begin{thm}\label{arb-large-thm}
  Let $\K$ be an $\LS (\K)$-tame AEC with weak amalgamation and coherent sequences, and let $\mu > \LS (\K)$ be a strong limit cardinal. \emph{If}:

  \begin{enumerate}
  \item Not every element of $\K_{\mu}$ is maximal.
  \item For every $\theta < \mu$ there exists $\lambda < \mu$ such that $\lambda = \lambda^{\theta}$ and $\K$ is categorical in $\lambda$.
  \end{enumerate}

  \emph{Then} $\K_{\ge \mu}$ has amalgamation and arbitrarily large models. 
\end{thm}
\begin{proof}
  Assume without loss of generality that $\mu$ has countable cofinality (otherwise, it is easy to find a smaller cardinal which has countable cofinality and still satisfies the hypotheses). Fix a non-maximal $M \in \K_\mu$. By Lemma \ref{main-lem-frame}, there is a $\goodms{S}$ $\mu$-frame on $\K_M$. It is easy to check that $\K_M$ has weak amalgamation and is $\mu$-tame. Therefore by Fact \ref{frame-upward}, $\K_M$ has amalgamation and arbitrarily large model. In particular, $\K$ has arbitrarily large models and $M$ is an amalgamation base in $\K$. Since any maximal model is an amalgamation base, we deduce that $\K_{\ge \mu}$ has amalgamation.
\end{proof}

We now work toward transferring categoricity assuming arbitrarily large models. We will use the following upward categoricity transfer of Grossberg and VanDieren for tame AECs with both amalgamation and arbitrarily large models.

\begin{fact}[{\cite[6.3]{tamenessthree}}]\label{gv-upward}
  Let $\K$ be an $\LS (\K)$-tame AEC with amalgamation and arbitrarily large models. If $\K$ is categorical in $\LS (\K)$ and in $\LS (\K)^+$, then $\K$ is categorical everywhere above $\LS (\K)$.
\end{fact}

Toward deriving global amalgamation, we show how to build a good frame assuming arbitrarily large models and enough categoricity, tameness, and amalgamation in low cardinals. We will use recent results from \cite{shvi-notes-apal} and \cite{categ-saturated-afml}, but the reader can regard them as black boxes. 

\begin{lem}\label{good-frame-categ}
  Let $\K$ be an AEC. \emph{If}:

  \begin{enumerate}
  \item $\K$ has arbitrarily large models.
  \item $\K$ is categorical in $\LS (\K)$.
  \item $\K$ is categorical in $\LS (\K)^+$.
  \item $\K_{\LS (\K)}$ has amalgamation.
  \item $\K_{\LS (\K)^+}$ has weak amalgamation.
  \item $\K_{\le \LS (\K)^+}$ is $\LS (\K)$-tame.
  \end{enumerate}

  Then $\K$ has a good $\LS (\K)^+$-frame.
\end{lem}
\begin{proof}
  Let $\lambda := \LS (\K)$. By \cite{shvi-notes-apal}, $\K$ is $\lambda$-superstable (essentially, this means that splitting has what is called the $\aleph_0$-local character for universal chains in \cite{ss-tame-jsl}) and by \cite[5.7(1)]{categ-saturated-afml}, $\K$ has $\lambda$-symmetry. We now attempt to build a $\goodms{S}$-frame as in \cite{ss-tame-jsl}. We can prove extension without using amalgamation in $\lambda^+$, hence by Fact \ref{type-ext-weak-ap} we obtain amalgamation in $\lambda^+$, and hence obtain a $\goodms{S}$ $\lambda^+$-frame. Symmetry is then proven as in \cite[6.8]{tame-frames-revisited-jsl} (and is not needed for the present paper).
\end{proof}

We can now prove a generalization of Fact \ref{gv-upward} for tame AECs with only weak amalgamation (but still with arbitrarily large models).

\begin{thm}\label{gv-weak-ap}
  Let $\K$ be an $\LS (\K)$-tame AEC with weak amalgamation and arbitrarily large models. If $\K_{\LS (\K)}$ has amalgamation, and $\K$ is categorical in both $\LS (\K)$ and $\LS (\K)^+$, then $\K_{\ge \LS (\K)}$ has amalgamation and $\K$ is categorical everywhere above $\LS (\K)$.
\end{thm}
\begin{proof}
  By Lemma \ref{good-frame-categ}, $\K$ has a good $\LS (\K)^+$-frame. By Fact \ref{frame-upward}, it follows that $\K_{\ge \LS (\K)^+}$ (and hence $\K_{\ge \LS (\K)}$) has amalgamation. Now apply Fact \ref{gv-upward}.
\end{proof}

Lemma \ref{good-frame-categ} still asked for full amalgamation in one cardinal. However we can use the weak diamond to derive it from successive categoricity:

\begin{cor}\label{gv-weak-ap-2}
  Let $\K$ be an $\LS (\K)$-tame AEC with weak amalgamation and arbitrarily large models. Let $\theta > \LS (\K)$ be least such that $2^{\LS (\K)} < 2^\theta$. If $\K$ is categorical in every cardinal in $[\LS (\K), \theta]$, then $\K$ is categorical everywhere above $\LS (\K)$. Moreover, there exists $\lambda \in [\LS (\K), \theta)$ such that $\K_{\ge \lambda}$ has amalgamation
\end{cor}
\begin{proof}
  Generalizing \cite[I.3.8]{shelahaecbook} (using the corresponding generalization of the Devlin-Shelah weak diamond \cite{dvsh65} as in \cite[1.2.4]{shvi635}), we can find $\lambda \in [\LS (\K), \theta)$ such that $\K_\lambda$ has amalgamation. Now apply Theorem \ref{gv-weak-ap} to $\K_{\ge \lambda}$.
\end{proof}

Putting together all the results proven so far, we obtain the main result of the paper:

\begin{cor}\label{main-cor}
  Let $\K$ be an $\LS (\K)$-tame AEC with weak amalgamation and coherent sequences. If $\K$ is categorical in every cardinal in $[\LS (\K), \beth_\omega (\LS (\K)))$, then $\K$ is categorical everywhere above $\LS (\K)$.
\end{cor}
\begin{proof}
  Let $\mu := \beth_\omega (\LS (\K))$. By Theorem \ref{basic-existence-thm}, not every element of $\K_\mu$ is maximal. By Theorem \ref{arb-large-thm}, $\K$ has arbitrarily large models. Now apply Corollary \ref{gv-weak-ap-2}.
\end{proof}

Specializing to universal classes, we get:

\begin{cor}\label{main-cor-2}
  If a universal $\Ll_{\omega_1, \omega}$ sentence is categorical in an end segment of cardinals strictly below $\beth_\omega$, then it is also categorical everywhere above $\beth_\omega$.
\end{cor}
\begin{proof}
  Let $\phi$ be a universal $\Ll_{\omega_1, \omega}$ sentence and let $\K$ be its the class of models (ordered with substructure). Note that $\K$ has intersections (the closure operator computed inside $N$ is the closure under the functions of $N$). By Lemma \ref{coherent-lem}, $\K$ has coherent sequences. By Remark \ref{tp-intersec}, $\K$ has weak amalgamation. By \cite[3.7]{ap-universal-apal}, $\K$ is $\LS (\K)$-tame. Let $\chi \in [\aleph_0, \beth_\omega)$ be such that $\K$ is categorical in every cardinal in $[\chi, \beth_\omega)$. Now apply Corollary \ref{main-cor} to $\K_{\ge \chi}$.
\end{proof}

For completeness, we also point out that the last two results can be drastically improved assuming the weak GCH (see also \cite{mv-universal-jsl} for how to improve even more when the Löwenheim-Skolem-Tarski number is $\aleph_0$). On $\mu_{\text{unif}}$, see \cite[VII.0.4]{shelahaecbook2} for a definition and \cite[VII.9.4]{shelahaecbook2} for what is known. It seems that for all practical purposes the reader can take $\mu_{\text{unif}} (\lambda^{++}, 2^{\lambda^+})$ to mean $2^{\lambda^{++}}$.

\begin{thm}
  Let $\K$ be an AEC with Löwenheim-Skolem-Tarski number $\lambda$. Assume that $\K$ is $\lambda^+$-tame and has weak amalgamation. Assume further that $2^{\lambda} < 2^{\lambda^+} < 2^{\lambda^{++}}$. If $\K$ is categorical in $\lambda$, $\K$ is categorical in $\lambda^+$, and $1 \le \Ii (\K, \lambda^{++}) < \mu_{\text{unif}} (\lambda^{++}, 2^{\lambda^+})$, then $\K$ is categorical everywhere above $\lambda$.
\end{thm}
\begin{proof}
  By results of Shelah on building good frames, there is a good $\lambda$-frame on $\K$ and $\K_{\le \lambda^+}$ is $\lambda$-tame (see \cite[7.1]{tame-succ-ijm} for an outline of the proof). By Fact \ref{frame-upward}, $\K$ has arbitrarily large models. Now apply Corollary \ref{gv-weak-ap-2}.
\end{proof}

\bibliographystyle{amsalpha}
\bibliography{existence-categ}

\providecommand{\bysame}{\leavevmode\hbox to3em{\hrulefill}\thinspace}
\providecommand{\MR}{\relax\ifhmode\unskip\space\fi MR }
\providecommand{\MRhref}[2]{%
  \href{http://www.ams.org/mathscinet-getitem?mr=#1}{#2}
}
\providecommand{\href}[2]{#2}
\begin{thebibliography}{BGVV17}

\bibitem[ABV19]{abv-categ-multi-apal}
Nathanael Ackerman, Will Boney, and Sebastien Vasey, \emph{Categoricity in
  multiuniversal classes}, Annals of Pure and Applied Logic \textbf{170}
  (2019), no.~11, 102712.

\bibitem[Bal09]{baldwinbook09}
John~T. Baldwin, \emph{Categoricity}, University Lecture Series, vol.~50,
  American Mathematical Society, 2009.

\bibitem[BGVV17]{shvi-notes-apal}
Will Boney, Rami Grossberg, Monica VanDieren, and Sebastien Vasey,
  \emph{Superstability from categoricity in abstract elementary classes},
  Annals of Pure and Applied Logic \textbf{168} (2017), no.~7, 1383--1395.

\bibitem[BK09]{bk-hs}
John~T. Baldwin and Alexei Kolesnikov, \emph{Categoricity, amalgamation, and
  tameness}, Israel Journal of Mathematics \textbf{170} (2009), 411--443.

\bibitem[Bon14]{ext-frame-jml}
Will Boney, \emph{Tameness and extending frames}, Journal of Mathematical Logic
  \textbf{14} (2014), no.~2, 1450007.

\bibitem[BS08]{non-locality}
John~T. Baldwin and Saharon Shelah, \emph{Examples of non-locality}, The
  Journal of Symbolic Logic \textbf{73} (2008), 765--782.

\bibitem[BV17a]{bv-survey-bfo}
Will Boney and Sebastien Vasey, \emph{A survey on tame abstract elementary
  classes}, Beyond first order model theory (Jos{\'e} Iovino, ed.), CRC Press,
  2017, pp.~353--427.

\bibitem[BV17b]{tame-frames-revisited-jsl}
\bysame, \emph{Tameness and frames revisited}, The Journal of Symbolic Logic
  \textbf{82} (2017), no.~3, 995--1021.

\bibitem[BV19]{logic-intersection-bpas}
\bysame, \emph{Structural logic and abstract elementary classes with
  intersections}, Bulletin of the Polish academy of science (mathematics)
  \textbf{67} (2019), 1--17.

\bibitem[DS78]{dvsh65}
Keith~J. Devlin and Saharon Shelah, \emph{A weak version of ${\Diamond}$ which
  follows from $2^{\aleph_0} < 2^{\aleph_1}$}, Israel Journal of Mathematics
  \textbf{29} (1978), no.~2, 239--247.

\bibitem[Gro02]{grossberg2002}
Rami Grossberg, \emph{Classification theory for abstract elementary classes},
  Contemporary Mathematics \textbf{302} (2002), 165--204.

\bibitem[GV06a]{tamenessthree}
Rami Grossberg and Monica VanDieren, \emph{Categoricity from one successor
  cardinal in tame abstract elementary classes}, Journal of Mathematical Logic
  \textbf{6} (2006), no.~2, 181--201.

\bibitem[GV06b]{tamenessone}
\bysame, \emph{Galois-stability for tame abstract elementary classes}, Journal
  of Mathematical Logic \textbf{6} (2006), no.~1, 25--49.

\bibitem[GV06c]{tamenesstwo}
\bysame, \emph{Shelah's categoricity conjecture from a successor for tame
  abstract elementary classes}, The Journal of Symbolic Logic \textbf{71}
  (2006), no.~2, 553--568.

\bibitem[HS90]{hs-example}
Bradd Hart and Saharon Shelah, \emph{Categoricity over ${P}$ for first order
  ${T}$ or categoricity for $\phi \in {L}_{\omega_1, \omega}$ can stop at
  $\aleph_k$ while holding for $\aleph_0, \ldots, \aleph_{k - 1}$}, Israel
  Journal of Mathematics \textbf{70} (1990), 219--235.

\bibitem[MAV18]{mv-universal-jsl}
Marcos Mazari-Armida and Sebastien Vasey, \emph{Universal classes near
  $\aleph_1$}, The Journal of Symbolic Logic \textbf{83} (2018), no.~4,
  1633--1643.

\bibitem[Ros97]{rosicky-sat-jsl}
Jiř{\'{\i}} Rosický, \emph{Accessible categories, saturation and
  categoricity}, The Journal of Symbolic Logic \textbf{62} (1997), no.~3,
  891--901.

\bibitem[She83a]{sh87a}
Saharon Shelah, \emph{Classification theory for non-elementary classes {I}: The
  number of uncountable models of $\psi \in {L}_{\omega_1, \omega}$. {P}art
  {A}}, Israel Journal of Mathematics \textbf{46} (1983), no.~3, 214--240.

\bibitem[She83b]{sh87b}
\bysame, \emph{Classification theory for non-elementary classes {I}: The number
  of uncountable models of $\psi \in {L}_{\omega_1, \omega}$. {P}art {B}},
  Israel Journal of Mathematics \textbf{46} (1983), no.~4, 241--273.

\bibitem[She87a]{sh88}
\bysame, \emph{Classification of non elementary classes {II}. {A}bstract
  elementary classes}, Classification Theory (Chicago, IL, 1985) (John~T.
  Baldwin, ed.), Lecture Notes in Mathematics, vol. 1292, Springer-Verlag,
  1987, pp.~419--497.

\bibitem[She87b]{sh300-orig}
\bysame, \emph{Universal classes}, Classification theory (Chicago, IL, 1985)
  (John~T. Baldwin, ed.), Lecture Notes in Mathematics, vol. 1292,
  Springer-Verlag, 1987, pp.~264--418.

\bibitem[She99]{sh394}
\bysame, \emph{Categoricity for abstract classes with amalgamation}, Annals of
  Pure and Applied Logic \textbf{98} (1999), no.~1, 261--294.

\bibitem[She01]{sh576}
\bysame, \emph{Categoricity of an abstract elementary class in two successive
  cardinals}, Israel Journal of Mathematics \textbf{126} (2001), 29--128.

\bibitem[She09a]{shelahaecbook}
\bysame, \emph{Classification theory for abstract elementary classes}, Studies
  in Logic: Mathematical logic and foundations, vol.~18, College Publications,
  2009.

\bibitem[She09b]{shelahaecbook2}
\bysame, \emph{Classification theory for abstract elementary classes 2},
  Studies in Logic: Mathematical logic and foundations, vol.~20, College
  Publications, 2009.

\bibitem[SV]{multidim-v2}
Saharon Shelah and Sebastien Vasey, \emph{Categoricity and multidimensional
  diagrams}, Preprint. URL: \url{ https://arxiv.org/abs/1805.06291v2}.

\bibitem[SV99]{shvi635}
Saharon Shelah and Andr{\'e}s Villaveces, \emph{Toward categoricity for classes
  with no maximal models}, Annals of Pure and Applied Logic \textbf{97} (1999),
  1--25.

\bibitem[Van06]{vandierennomax}
Monica VanDieren, \emph{Categoricity in abstract elementary classes with no
  maximal models}, Annals of Pure and Applied Logic \textbf{141} (2006),
  108--147.

\bibitem[Vas16a]{ss-tame-jsl}
Sebastien Vasey, \emph{Forking and superstability in tame {A}{E}{C}s}, The
  Journal of Symbolic Logic \textbf{81} (2016), no.~1, 357--383.

\bibitem[Vas16b]{sv-infinitary-stability-afml}
\bysame, \emph{Infinitary stability theory}, Archive for Mathematical Logic
  \textbf{55} (2016), 567--592.

\bibitem[Vas17a]{categ-saturated-afml}
\bysame, \emph{Saturation and solvability in abstract elementary classes with
  amalgamation}, Archive for Mathematical Logic \textbf{56} (2017), 671--690.

\bibitem[Vas17b]{ap-universal-apal}
\bysame, \emph{Shelah's eventual categoricity conjecture in universal classes:
  part {I}}, Annals of Pure and Applied Logic \textbf{168} (2017), no.~9,
  1609--1642.

\bibitem[Vas17c]{categ-universal-2-selecta}
\bysame, \emph{Shelah's eventual categoricity conjecture in universal classes:
  part {I}{I}}, Selecta Mathematica \textbf{23} (2017), no.~2, 1469--1506.

\bibitem[Vas18]{stab-spec-jml}
\bysame, \emph{Toward a stability theory of tame abstract elementary classes},
  Journal of Mathematical Logic \textbf{18} (2018), no.~2, 1850009.

\bibitem[Vas19]{categ-amalg-selecta}
\bysame, \emph{The categoricity spectrum of large abstract elementary classes
  with amalgamation}, Selecta Mathematica \textbf{25} (2019), no.~5, 65.

\bibitem[Vas20]{tame-succ-ijm}
\bysame, \emph{Tameness from two successive good frames}, Israel Journal of
  Mathematics \textbf{235} (2020), 465--500.

\bibitem[Zil05]{zil05}
Boris Zilber, \emph{A categoricity theorem for quasi-minimal excellent
  classes}, Logic and its applications (Andreas Blass and Yi~Zhang, eds.),
  Contemporary Mathematics, American Mathematical Society, 2005, pp.~297--306.

\end{thebibliography}

\end{document}